\documentclass[12pt]{article}
\usepackage{epsf}
\usepackage{amsbsy,amsmath}
\usepackage{mathtools}
\usepackage{mathrsfs}
\usepackage{amsfonts}
\usepackage{amssymb}
\usepackage{enumitem}
\usepackage{eucal}
\usepackage{graphics,mathrsfs}
\usepackage{amsthm}
\usepackage{secdot}
\newtheorem{theorem}{Theorem}[section]
\newtheorem{lemma}[theorem]{Lemma}

\theoremstyle{definition}

\theoremstyle{remark}

\numberwithin{equation}{section}

\numberwithin{equation}{section}

\newsavebox{\savepar}

\begin{document}
\title{\bf{{\Large A study on elliptic PDE involving $p$-harmonic and $p$-biharmonic operator with steep potential well}} }
\author{Ratan Kr Giri, Debajyoti Choudhuri \footnote{Corresponding author, Date: \today.}\,\, and Shesadev Pradhan \\
 {\it{\small Department of Mathematics, National Institute of Technology, Rourkela, India }}\\
{\it{\small e-mails: 512MA1007@nitrkl.ac.in,
choudhurid@nitrkl.ac.in, pradhansh@nitrkl.ac.in}}}

\date{}
\maketitle
\begin{abstract}
\noindent In this paper, we give an existence result pertaining to a
nontrivial solution to the problem
\begin{align*}
\Bigg\{\begin{split}
& \Delta^2_p u -\Delta_p u + \lambda V(x)|u|^{p-2}u  = f(x,u)\,,\,x\in \mathbb{R}^N,\\
& u \in W^{2,p}(\mathbb{R}^N),
\end{split}
\end{align*}
where $p>1$, $\lambda>0$, $V\in C(\mathbb{R}^N,\mathbb{R}^+)$, $f\in
C(\mathbb{R}^N \times \mathbb{R},\mathbb{R})$, $N>2p$. We also
explore the problem in the limiting case of $\lambda \rightarrow
\infty$.
\end{abstract}
\section{Introduction}
The problem we will address in this article is \begin{align}
\Bigg\{\begin{split}
& \Delta^2_p u -\Delta_p u + \lambda V(x)|u|^{p-2}u  = f(x,u)\,,\,x\in \mathbb{R}^N,\\
& u \in W^{2,p}(\mathbb{R}^N), \label{ineq1}
\end{split}
\end{align}
where $\Delta^2_ p u = \Delta(|\Delta u|^{p-2}\Delta u)$, $\Delta_p
u = \nabla \cdot (|\nabla u|^{p-2}\nabla u)$ and $\lambda >0$ is a
parameter with $p>1$, $N>2p$. The potential function $V(x)$ is a
real valued continuous function on $\mathbb{R}^N$ satisfying the
following conditions:
\begin{align*}
\begin{split}
 & \text{(V1)}\text{ $V(x)\geq 0$ on $\mathbb{R}^N$. } \\
& \text{(V2)}\text{ There exists $b>0$ such that the set $V_b = \{x\in \mathbb{R}^n: V(x)<b\}$ is}\\
& \text{nonempty and has finite Lebesgue measure in $\mathbb{R}^N$.}\\
& \text{(V3)} \text{\,\,$\Omega= \text{int}\, V^{-1}(0)$ is nonempty
and has smooth boundary with $\bar{\Omega}=V^{-1}(0)$.}
\end{split}
\end{align*}
This type of assumptions were introduced by Bartsch et al
\cite{pank} (see also \cite{ding}), referred to as the steep well
potential for the potential function $V(x)$, in the study of a
nonlinear Schr\"{o}dinger equation. Further, to study the existence
of nontrivial solution and the limiting case, $\lambda \rightarrow
\infty$, of the problem $\ref{ineq1}$, we make the following
assumptions on the nonlinear function  $f$:
\begin{align*}
\begin{split}
& \text{(F1)} \text{ $f\in C(\mathbb{R}^N \times
\mathbb{R},\mathbb{R})$ and there exists constants $1 <\gamma_1
<\gamma_2 <\cdots <\gamma_m <p$ and}\\
& \text{functions $\xi_i \in
L^{\frac{p}{p-\gamma_i}}(\mathbb{R}^N, \mathbb{R}^+)$ such that}\\
&\text {~~~~~~~~~~~~~~~~~~~~~~~~~~~~~~$|f(x, u)|\leq \sum_{i=1}^m
\gamma_i \xi_i(x) |u|^{\gamma_i-1},
\forall \,(x, u)\in \mathbb{R}^N\times \mathbb{R}.$}\\
 & \text{(F2)} \text{ There exists constants $\eta, \delta >0, \gamma_0
 \in
 (1, p)$ such that}\\
 &\text{~~~~~~~~~~~~~~~~~~~~~~~~~~~~~~$|F(x, u)| \geq \eta |u|^{\gamma_0}$ for all $x\in \Omega$ and for all such $u$ such that $|u|\leq
 \delta$,}\\
 & \text{ where $F(x, u)= \int_0 ^u f(x, s) ds.$}
\end{split}
\end{align*}
In the recent years problems related to the kind in $(\ref{ineq1})$,
for  the case of $p=2$, the equations with biharmonic operator have
been investigated. Readers may refer to \cite{liu, we, shen, zhang,
wei, yin, ye1, yang, ye} and the references there in. The present
work in this article draws its motivation from W. Zhang et al
\cite{zhang}, Ye \& Tang \cite{ye} and Liu et al \cite{liu}. In all
these articles, they have considered the problem $(\ref{ineq1})$ for
$p=2$. We address the problem for $p\neq 2$, $N>2p$. The notion of
$p$-biharmonic operator is introduced in the recent work of Bhakta
\cite{bhakta}. For $p\neq 2$, things seemed to become more
complicated not only due to the lack of linearity of both
$p$-Laplacian and $p$-biharmonic operator but also because of the
fact that the associated energy functional is defined on a Banach
space which is not a Hilbert space. We further have to deal with
lack of compact embedding, since the domain considered here is
$\mathbb{R}^N$.
The main two results proved in this article the following.
\begin{theorem}\label{thm2}
Assume the conditions (V1)-(V3), (F1), (F2) to hold. Then there
exists $\Lambda_0>0$ such that for each $\lambda> \Lambda_0$,
problem $(\ref{ineq1})$ has at least one non trivial solution
$u_\lambda$.
\end{theorem}
\begin{theorem}\label{thm3}
Let $u_n = u_{\lambda_n}$ be a solution of the problem
$(\ref{ineq1})$ corresponding to $\lambda=\lambda_n$. If $\lambda_n
\rightarrow \infty$, then
$$||u_n||_{\lambda_n} \leq c,$$ for some $c>0$ and for $p\leq q <p_*$,  $$u_n \rightarrow
\tilde{u}\,\,\mbox{in}\,\,L^q(\mathbb{R}^N)$$ up to a subsequence.
Further, this $\tilde{u}$ is a solution of the problem
\begin{align}
\begin{split}
\Delta ^2_p u - \Delta_p u & = f(x, u),\,\, \mbox{in}\,\,\Omega\\
u & =0,\,\,\mbox{on}\,\,\mathbb{R}^N\setminus \Omega. \label{ineq7}
\end{split}
\end{align}
and $u_n \rightarrow \tilde{u}$ in $W^{2,p}(\mathbb{R}^N)$.
\end{theorem}
The paper has been organized as follows. In section 2, we discuss
the notations which will be used in the theorems. In section 3, we
give the proof of Theorem $\ref{thm2}$ and in section 4, we prove
the Theorem $\ref{ineq7}$.
\section{Preliminaries and Notations}
We will denote a Sobolev space of order $2$ as
$W^{2,p}(\mathbb{R}^N)$, which is given by
$$W^{2,p}(\mathbb{R}^N) =\{ u\in L^p(\mathbb{R}^N): |\nabla u|,
\Delta u \in L^p(\mathbb{R}^N)\}$$ endowed with the norm
$$ ||u||_{W^{2,p}(\mathbb{R}^N)} ^p =
\int_{\mathbb{R}^N}(|\Delta u|^p + |\nabla u|^p + |u|^p)dx.$$  Let
$$X = \left\{ u\in W^{2,p}(\mathbb{R}^N): \int_{\mathbb{R}^n}(|\Delta
u|^p + |\nabla u|^p + V(x)|u|^p)dx <\infty\right\}$$ be endowed with
the norm
$$ ||u||^p =
\int_{\mathbb{R}^n}(|\Delta u|^p + |\nabla u|^p + V(x)|u|^p)dx.$$
For $\lambda>0$, we set
$$E_\lambda = \{ u\in W^{2,p}(\mathbb{R}^N): \int_{\mathbb{R}^N}(|\Delta
u|^p + |\nabla u|^p + \lambda V(x)|u|^p)dx <\infty\}$$ with
$$||u||^p_\lambda = \int_{\mathbb{R}^N}(|\Delta u|^p + |\nabla u|^p +
\lambda V(x)|u|^p)dx.$$ It is easy to verify that $(E_\lambda, ||
\cdot||_\lambda)$ is a closed in $W^{2,p}(\mathbb{R}^N)$ and
$$||u|| \leq ||u||_\lambda$$ for any $\lambda\geq
1$. We will denote $\mu$ to be the Lebesgue measure on
$\mathbb{R}^N$.
\begin{lemma}\label{lem3}
If \text{(V1)-(V2)} hold, then there exists positive constants
$\lambda_0, c_0$ such that $$||u||_{W^{2,p}(\mathbb{R}^N)} \leq c_0
||u||_\lambda;\,\,\,\mbox{for\,\,all}\,\,u\in E_\lambda, \lambda\geq
\lambda_0.$$
\end{lemma}
\begin{proof}
By using conditions (V1)-(V2) and the Sobolev inequality, we have
\begin{align*}
||u||_{W^{2,p}(\mathbb{R}^N)} & = \int_{\mathbb{R}^N}(|\Delta u|^p
+ |\nabla u|^p + |u|^p)dx \\
& =  \int_{\mathbb{R}^N}(|\Delta u|^p + |\nabla u|^p )dx +
\int_{V_b}|u|^pdx + \int_{\mathbb{R}^N\setminus V_b} |u|^pdx\\
& \leq \int_{\mathbb{R}^N}(|\Delta u|^p + |\nabla u|^p )dx +
\left(\mu(V_b)\right)^{\frac{P^* -
p}{p*}}\left(\int_{\mathbb{R}^N}|u|^{p^*}dx\right)^{\frac{p}{p^*}}
+\int_{\mathbb{R}^N\setminus V_b} |u|^pdx\\
& \leq \int_{\mathbb{R}^N}(|\Delta u|^p + |\nabla u|^p )dx +
S_{\alpha}^{-1}\left(\mu(V_b)\right)^{\frac{P^* -
p}{p*}}\int_{\mathbb{R}^N} |\nabla u|^p dx + \frac{1}{\lambda
b}\int_{\mathbb{R}^N\setminus
V_b} \lambda V(x)|u|^pdx\\
& \leq \int_{\mathbb{R}^N}(|\Delta u|^p + (1+
S_{\alpha}^{-1}\left(\mu(V_b)\right)^{\frac{P^* - p}{p*}})|\nabla
u|^p )dx
+\frac{1}{\lambda b}\int_{\mathbb{R}^N} V(x)|u|^pdx\\
&\leq \max \left\{1,
1+S_{\alpha}^{-1}\left(\mu(V_b)\right)^{\frac{P^* - p}{p*}},
\frac{1}{\lambda b}\right\} \int_{\mathbb{R^N}} (|\Delta u|^p +
|\nabla u|^p + \lambda V(x)|u|^p)dx.
\end{align*}
where $S_\alpha$ denote the Sobolev constant, $p^* =
\frac{Np}{N-p}$. Take $\lambda_0 = \frac{1}{b \left(1+
S_\alpha^{-1}(\mu(V_b))^{\frac{P^* - p}{p*}}\right)}$. Then for all
$\lambda\geq \lambda_0$, we have
\begin{align*}
 \max \left\{1, 1+S_\alpha^{-1}\left(\mu(V_b)\right)^{\frac{P^* - p}{p*}},
\frac{1}{\lambda b}\right\} & =\max \left\{1,
1+S_\alpha^{-1}\left(\mu(V_b)\right)^{\frac{P^* -
p}{p*}}\right\}\\
& = c_0\,\,\mbox{(say)}.
\end{align*}
Hence for all $\lambda \geq \lambda_0$  and $u \in E_\lambda$, we
have $||u||_{W^{2,p}(\mathbb{R}^N)} \leq c_0 ||u||_\lambda.$
\end{proof}
The lemma shows that $E_\lambda \hookrightarrow
W^{2,p}(\mathbb{R}^N)$. By the Sobolev embedding results for $p<N$
we have $W^{2,p}(\mathbb{R}^N)\hookrightarrow L^q(\mathbb{R^N})$,
for $q\in [p, p_*]$. Hence there exists $c_q>0$ such that
$$||u||_q \leq c_q ||u||_{W^{2,p}(\mathbb{R}^N)} \leq c_0c_q
||u||_\lambda,$$ for all $\lambda \geq \lambda_0$, $q\in [p, p_*]$.
\section{Existence of non trivial solutions}

Let $$J_\lambda (u) = \frac{1}{p}\int_{\mathbb{R}^N} (|\Delta u|^p +
|\nabla u|^p + \lambda V(x) |u|^p)dx - \int_{\mathbb{R}^N} F(x,
u)dx.$$ Then it can be seen that $J_\lambda \in C^1(E_\lambda,
\mathbb{R})$ and its Fr\'{e}chet derivative is given by
$$<J_\lambda '(u), v> = \int_{\mathbb{R}^N}(|\Delta u|^{p-2}\Delta u\Delta v + |\nabla u|^{p-2}\nabla u\cdot \nabla v + \lambda V(x) |u|^{p-2}u v)dx
-\int_{\mathbb{R}^N} f(x, u)vdx,$$ for all $v\in E_\lambda$. Thus
$u\in E_\lambda$ is a critical point of $J_\lambda$ if and only if
it is a weak solution of the problem ($\ref{ineq1}$). In order to
prove the existence of non trivial solutions of the problem
($\ref{ineq1}$) we use the following theorem \cite{rabi}
\begin{theorem}\label{thm1}.
Let $B$ be a real Banach space and $J\in C^1(B, \mathbb{R})$ satisfy
the Palais-Smale (PS) condition. If $J$ is bounded below, then
$\displaystyle {c= \inf _B J}$ is a critical value of $J$.
\end{theorem}
We now prove the following lemmas.
\begin{lemma}\label{lem1}
Suppose that (V1)-(V3), (F1),(F2) are satisfied. Then there exists
$\Lambda_0 >0$ such that for every $\lambda\geq \Lambda_0$,
$J_\lambda$ is bounded below in $E_\lambda$.
\end{lemma}
\begin{proof}
Using the H\"{o}lder's inequality we have,
\begin{align*}
J_\lambda(u) & = \frac{1}{p} ||u||_\lambda ^p
-\int_{\mathbb{R}^N}F(x,
u)dx\\
& \geq \frac{1}{p}||u||_\lambda ^p - \sum_{i=1}^m
\int_{\mathbb{R}^N} \xi_i(x) |u|^{\gamma_i} dx\\
& \geq \frac{1}{p}||u||_\lambda ^p - \sum_{i=1}^m
\left(\int_{\mathbb{R}^N}
|\xi_i(x)|^{\frac{p}{p-\gamma_i}}dx\right)^{\frac{p-\gamma_i}{p}}\left(\int_{\mathbb{R}^N}|u|^pdx\right)^{\frac{\gamma_i}{p}}\\
& \geq \frac{1}{p} ||u||_\lambda ^p - \sum_{i=1}^m
c_p^{\gamma_i}c_0^{\gamma_i}
||\xi||_{\frac{p}{p-\gamma_i}}||u||_\lambda ^{\gamma_i}.
\end{align*}
Since $1<\gamma_1<\cdots <\gamma_m <p$, the above inequality implies
that $J_\lambda(u) \rightarrow +\infty$, when $||u||_\lambda
\rightarrow +\infty$. Consequently, there exists $\Lambda_0 = \max
\{1, \lambda_0\}>0$ such that for every $\lambda \geq \Lambda_0$,
$J_\lambda$ is bounded from below.
\end{proof}
\begin{lemma}\label{lem2}
Assume that the conditions (V1)-(V3), (F1), (F2) are satisfied. Then
$J_\lambda$ satisfies the Palais-Smale (PS) condition for every
$\lambda \geq \Lambda_0$.
\end{lemma}
\begin{proof}
Suppose that $(u_n) \subset E_\lambda$ be a sequence such that
$J_\lambda (u_n)$ is bounded and $J_\lambda '(u_n)\rightarrow 0$ as
$n\rightarrow \infty$. Then by lemma $\ref{lem1}$, $(u_n)$ is
bounded below in $E_\lambda$. Thus there exists a constant $c>0$
such that for all $n\in \mathbb{N}$,
$$||u_n||_q \leq c_q c_0 ||u_n||_\lambda \leq c;
\,\,\mbox{for\,\,all}\,\,u\in E_\lambda, \lambda \geq \lambda_0,
p\leq q\leq p_*.$$ Hence by Eberlein-Smulian theorem, passing on to
a subsequence (the subsequence is still denoted by $u_n$), we may
assume that $u_n \rightharpoonup u_0$ in $E_\lambda$. Since the
inclusion $E_\lambda \hookrightarrow L^q_{loc}(\mathbb{R}^N)$ is
compact for $q \in [p, p_*)$, we have
$$u_n \rightarrow u_0\,\,\,\mbox{in}\,\,L^p_{loc}(\mathbb{R}^N).$$
Since $\xi_i(x) \in L^{\frac{p}{p-\gamma_i}}(\mathbb{R}^N,
\mathbb{R}^+)$, we can choose $R_\epsilon >0$ such that
\begin{eqnarray}
\left(\int_{\mathbb{R}^N\setminus B_{R_\epsilon}}
|\xi_i(x)|^{\frac{p}{p-\gamma_i}}dx\right)^{\frac{p-\gamma_i}{p}} <
\epsilon,\,\,\,1\leq i\leq m.\label{ineq2}
\end{eqnarray}
Since $u_n \rightarrow u_0$ in $L^p_{loc}(\mathbb{R}^N)$, there
exists $N_0\in \mathbb{N}$ such that
\begin{eqnarray}
\left(\int_{B_{R_\epsilon}}
|u_n-u_0|^pdx\right)^{\frac{\gamma_i}{p}} <\epsilon \label{ineq3}
\end{eqnarray}
for $n\geq N_0$ and for all $1\leq i\leq m$. By $(\ref{ineq3})$ and
H\"{o}lder inequality, we have,
\begin{align*}
 \int_{B_{R_\epsilon}} |f(x, u_n-u_0)||u_n-u_0|dx & \leq
\sum_{i=1}^m \gamma_i \int_{B_{R_\epsilon}}
|\xi_i(x)||u_n-u_0|^{\gamma_i} dx\\
& \leq \sum_{i=1}^m \gamma_i \left(\int_{B_{R_\epsilon}}
|\xi_i(x)|^{\frac{p}{p-\gamma_i}}dx\right)^{\frac{p-\gamma_i}{p}}\left(\int_{B_{R_\epsilon}}
|u_n-u_0|^pdx\right)^{\frac{\gamma_i}{p}}\\
& \leq \left(\sum_{i=1}^m \gamma_i
||\xi_i||_{\frac{p}{p-\gamma_i}}\right)\epsilon, \forall n \geq N_0.
\end{align*}
Hence it follows that
\begin{eqnarray}
\int_{B_{R_\epsilon}} |f(x, u_n-u_0)||u_n-u_0|dx \rightarrow 0,
\,\,\,\mbox{as}\,\,n\rightarrow \infty. \label{ineq4}
\end{eqnarray}
On the other hand, by $(\ref{ineq2})$ and boundedness of $(u_n)$ in
$L^p(\mathbb{R}^N)$ we have,
\begin{align*}
\int_{\mathbb{R}^N\setminus B_{R_\epsilon}}|f(x,
u_n-u_0)||u_n-u_0|dx & \leq \sum_{i=1}^m \gamma_i
||\xi_i||_{\frac{p}{p-\gamma_i},\mathbb{R}^N\setminus
B_{R_\epsilon}} ||u_n-u_0||_{p, \mathbb{R}^N\setminus
B_{R_\epsilon}}^{\gamma_i}\\
& \leq \epsilon \sum_{i=1}^m \gamma_i ||u_n -u_0||_p^{\gamma_i}\\
& \leq \epsilon \sum_{i=1}^m \gamma_i \left(||u_n||_p +
||u_0||_p\right)^\gamma_i\\
&\leq \epsilon \sum_{i=1}^m \gamma_i \left(c +
||u_0||_p\right)^{\gamma_i}.
\end{align*}
Therefore,
\begin{eqnarray}
\int_{\mathbb{R}^N\setminus B_{R_\epsilon}}|f(x,
u_n-u_0)||u_n-u_0|dx \rightarrow 0,\,\,\mbox{as}\,\,n\rightarrow
\infty. \label{ineq5}
\end{eqnarray}
Combining $(\ref{ineq4})$ and $(\ref{ineq5})$, we have
\begin{eqnarray}
\int_{\mathbb{R}^N} |f(x, u_n-u_0)||u_n-u_0|dx \rightarrow
0,\,\,\mbox{as}\,\,n\rightarrow \infty.\label{ineq6}
\end{eqnarray}
Since $u_n \rightharpoonup u_0$ in $E_\lambda$, hence
$\left<J'_\lambda(u_n -u_0), u_n-u_0\right> \rightarrow 0$ as
$n\rightarrow \infty$. But
\begin{align*}
0 \leq ||u_n -u_0||_\lambda ^p & =\left<J'_\lambda(u_n -u_0),
u_n-u_0\right> +
\int_{\mathbb{R}^N} (f(x, u_n-u_0))(u_n-u_0)dx\\
& \leq \,\left<J'_\lambda(u_n -u_0), u_n-u_0\right> +
\int_{\mathbb{R}^N} |f(x, u_n-u_0)||u_n-u_0|dx
\end{align*}
By $(\ref{ineq6})$ and $\left<J'_\lambda(u_n -u_0), u_n-u_0\right>
\rightarrow 0$, it follows that $||u_n -u_0||_\lambda ^p\rightarrow
0$ as $n\rightarrow \infty$. This shows that $u_n\rightarrow u_0$ in
$E_\lambda$.
\end{proof}
\begin{proof}[\textbf{Proof of the Theorem $\ref{thm2}$}]
By lemmas $\ref{lem1}$, $\ref{lem2}$ and theorem $\ref{thm1}$, it
follows that $c_\lambda =\inf_{E_\lambda} J_\lambda (u)$ is a
critical value of $J_\lambda$, that is there exists a critical point
$u_\lambda \in E_\lambda$ such that $J_\lambda(u_\lambda)=
c_\lambda$. Therefore, $u_\lambda$ is a solution for the problem
$(\ref{ineq1})$ for $\lambda >\Lambda_0$. Now we will show that
$u_\lambda \neq 0$. Let $u^{\#} \in \left(W^{2,p}(\Omega)\cap
W^{1,p}_0(\Omega)\right)\setminus\{0\}$ such that $||u^{\#}||_\infty
\leq 1$, where $\Omega$ is given in the condition (V3). Then by the
condition (F2), we have
\begin{align*}
J_\lambda(tu^{\#}) & = \frac{1}{p}||tu^{\#}||_\lambda ^p -
\int_{\mathbb{R}^N} F(x, tu^*)dx\\
& = \frac{t^p}{p}||u^{\#}||_\lambda^P - \int_\Omega F(x, tu^{\#})dx\\
& \leq \frac{t^p}{p} ||u^{\#}||_\lambda^p -\eta t^{\gamma_0}
\int_\Omega |u^{\#}|^{\gamma_0}dx,
\end{align*}
$\forall t\in(0,\delta)$, where $\delta$ is defined in (F2). Since
$1<\gamma_0 <p$, it follows that $J_\lambda(tu^{\#}) <0$ for $t>0$
small enough. Hence $J_\lambda(u_\lambda) =c_\lambda <0$. Therefore,
$u_\lambda$ is a nontrivial solution of the problem $(\ref{ineq1})$.
\end{proof}
\section{Limiting case $\lambda\rightarrow\infty$}
We consider the limiting case, $\lambda\rightarrow \infty$,  of the
problem $(\ref{ineq1})$ on the set $ V^{-1}(0)$. Define
$\tilde{W}(\Omega) = \begin{cases} u, u \in W^{2,p}(\Omega)\cap
W^{1,p}_0(\Omega)\\
u=0\,\, a.e.\,, \,\,\mbox{in}\,\, R^N\setminus\Omega.
\end{cases}$,\\ where $\Omega$ is given in the condition (V3).
Then $\tilde{W}(\Omega)\subset E_\lambda$ for all $\lambda >0$.
Define
$$\tilde{c}= \inf_{u\in \tilde{W}(\Omega)}
J_\lambda |_{\tilde{W}(\Omega)},$$ where $J_\lambda
|_{\tilde{W}(\Omega)}$ is a restriction of $J_\lambda$ on
$\tilde{W}(\Omega)$, that is
$$J_\lambda |_{\tilde{W}(\Omega)} =
\frac{1}{p}\int_\Omega \left(|\Delta u|^p + |\nabla u|^p\right)dx
-\int_\Omega F(x, u)dx, $$ for $u\in \tilde{W}(\Omega))$. Similar to
the proof of the theorem $\ref{thm2}$, it can be seen that
$\tilde{c}<0$ is achieved and
$$c_\lambda \leq \tilde{c} <0,\,\,\,\mbox{for\,\,all}\,\,\lambda
> \Lambda_0.$$
\begin{proof}[\textbf{Proof of the Theorem $\ref{thm3}$}]
For any sequence $\lambda_n\rightarrow \infty$, let $u_n =
u_{\lambda_n}$ be the critical points of $J_{\lambda_n}$. Thus we
have,
\begin{eqnarray}
c_{\lambda_n} = J_{\lambda_n}(u_n) \leq \tilde{c} <0. \label{ineq8}
\end{eqnarray}
Also in the lemma $\ref{lem1}$, we have seen that
$$J_{\lambda_n}(u_n) \geq \frac{1}{p} ||u_n||_{\lambda_n}^p
-\sum_{i=1}^m
c_p^{\gamma_i}c_0^{\gamma_i}||\xi_i||_{\frac{p}{p-\gamma_i}}||u_n||_{\lambda_n}^{\gamma_i}.$$
Therefore, $(\ref{ineq8})$ and the above inequality implies that
\begin{eqnarray}
||u_n||_{\lambda_n} \leq c, \label{ineq9}
\end{eqnarray}
where the constant $c>0$ is independent of $\lambda_n$. Therefore,
passing on to a subsequence we may assume that $u_n \rightharpoonup
\tilde{u}$ in $E_\lambda$. This implies that $u_n \rightarrow
\tilde{u}$ in $L^q_{loc}(\mathbb{R}^N)$ for $p\leq q <p_*$. Then by
Fatou's lemma we have, $$ \int_{\mathbb{R}^N} V(x) |\tilde{u}|^p dx
\leq \lim_{n\rightarrow \infty}\inf \int_{\mathbb{R}^N} V(x) |u_n|^p
dx \leq \lim_{n\rightarrow \infty}\inf
\frac{||u_n||_{\lambda_n}^p}{\lambda_n}=0,$$ which implies that
$\tilde{u}=0$ a.e. in $\mathbb{R}^N\setminus V^{-1}(0)$. Since
$\Omega =\text{int} V^{-1}(0)$ and $\Omega$ has smooth boundary,
hence $\tilde{u}\in \tilde{W}(\Omega)$. Now for any $\varphi \in
C_0^\infty(\Omega)$, since $\left<J'_{\lambda_n}(u_n),
\varphi\right>=0$, it follows that
$$\int_\Omega \left(|\Delta \tilde{u}|^{p-2}\Delta
\tilde{u}\Delta \varphi + |\nabla\tilde{u}|^{p-2}\nabla u\cdot
\nabla \varphi\right)dx - \int_\Omega f(x, \tilde{u})\varphi dx
=0,$$ which implies that $\tilde{u}$ is a weak solution of the
problem $(\ref{ineq7})$, where we have used the density of
$C_0^\infty(\Omega)$ in
$\tilde{W}(\Omega)$.\\
Next we show that $u_n \rightarrow \tilde{u}$ in
$L^q(\mathbb{R}^N)$. If not, then by Lions Vanishing lemma
\cite{lions, willem}, there exists $\delta>0$, $\rho>0$ and sequence
$\{x_n\} \in \mathbb{R}^N$ with $|x_n|\rightarrow \infty$ such that
$$\int_{B_\rho(x_n)} |u_n - \tilde{u}|^p dx \geq \delta.$$ Since
$|x_n|\rightarrow \infty $, hence $\mu\left(B_\rho(x_n)\cap
V_b\right)\rightarrow 0$ as $n \rightarrow \infty$. Therefore, by
H\"{o}lder inequality, we have
$$\int_ {B_\rho(x_n)\cap V_b} |u_n -
\tilde{u}|^p dx  \leq \mu\left(B_\rho(x_n)\cap
V_b\right)^{\frac{p_*-p}{p_*}}\left(\int_{\mathbb{R}^N} |u_n
-\tilde{u}|^{p_*}dx\right)^{\frac{p}{p_*}} \rightarrow 0,$$ as
$n\rightarrow \infty$. Consequently
\begin{align*}
||u_n||_{\lambda_n}^p & \geq \lambda_n b \int_{B_\rho(x_n)\cap\{x\in
\mathbb{R}^N: V(x)\geq b\}}|u_n|^p dx \\
& = \lambda_n b \int_{B_\rho(x_n)\cap\{x\in \mathbb{R}^N: V(x)\geq
b\}}|u_n- \tilde{u}|^p dx,\,\,\,[\mbox{as}\,\, \tilde{u}=0\,\,
\mbox{a.e.}\,\,\mbox{in}\,\,\,\,\mathbb{R}^N \setminus V^{-1}(0)]\\
& = \lambda_n b \left( \int_{B_\rho(x_n)} |u_n -\tilde{u}|^p dx -
\int_{B_\rho(x_n)\cap V_b}|u_n- \tilde{u}|^p dx + o(1)\right)\\
& \rightarrow \infty,\,\,\mbox{as}\,\,n\rightarrow \infty
\end{align*}
 which contradicts to $(\ref{ineq9})$. Therefore, $u_n \rightarrow \tilde{u}$ in
$L^q(\mathbb{R}^N)$ for $p\leq q<p_*$.\\
Next, we show that $u_n \rightarrow \tilde{u}$ in
$W^{2,p}(\mathbb{R}^N)$. Since $u_n \rightharpoonup \tilde{u}$ in
$E_\lambda$ and $E_\lambda$ compactly embedded in
$L^q_{loc}(\mathbb{R}^N)$, hence by the similar method as in lemma
$\ref{lem2}$, it follows that
\begin{eqnarray}
\int_{\mathbb{R}^N} |f(x, u_n-\tilde{u})||u_n-\tilde{u}|dx
\rightarrow 0,\,\,\mbox{as}\,\,n\rightarrow \infty.\label{ineq10}
\end{eqnarray}
Since $u_n \rightharpoonup \tilde{u}$ in $E_\lambda$, hence
$\left<J'_\lambda(u_n -\tilde{u}), u_n-\tilde{u}\right> \rightarrow
0$ as $n\rightarrow \infty$. But
\begin{align*}
||u_n -\tilde{u}||_\lambda ^p & =\left<J'_\lambda(u_n -\tilde{u}),
u_n-\tilde{u}\right> +
\int_{\mathbb{R}^N} (f(x, u_n-\tilde{u}))(u_n-\tilde{u})dx\\
& \leq \,\left<J'_\lambda(u_n -\tilde{u}), u_n-\tilde{u}\right> +
\int_{\mathbb{R}^N} |f(x, u_n-\tilde{u})||u_n-\tilde{u}|dx
\end{align*}
Therefore, by $(\ref{ineq10})$ and the above inequality, it follows
that $||u_n -\tilde{u}||_\lambda \rightarrow 0$ as $n\rightarrow
\infty$. Again by the lemma $\ref{lem3}$,
$$||u_n-\tilde{u}||_{W^{2,p}(\mathbb{R}^N)} \leq c_0
||u_n -\tilde{u}||_\lambda,$$ hence we have $u_n \rightarrow
\tilde{u}$ in $W^{2,p}(\mathbb{R}^N)$.\\
From $(\ref{ineq8})$, we have $$\frac{1}{p}\int_\Omega \left(|\Delta
\tilde{u}|^p + |\nabla \tilde{u}|^p\right)dx -\int_\Omega F(x,
\tilde{u})dx \leq \tilde{c} <0,$$ which implies that $\tilde{u}\neq
0$. This completes the theorem.
\end{proof}
\section*{Acknowledgement}
 One of the authors (R. Kr. Giri) thanks the financial assistantship received from
 the Ministry of Human Resource Development
 (M.H.R.D.), Govt. of India.

\bibliographystyle{amsplain}

\end{document}